\documentclass[12pt,leqno]{article}
\usepackage{hyperref} 
\usepackage{enumitem}
\usepackage{xcolor}
\usepackage{soul}
\usepackage[doc,com]{optional}
\usepackage[doc]{optional}
\definecolor{labelkey}{rgb}{0,0.08,0.45}
\definecolor{rekey}{rgb}{0,0.6,0.0}
\definecolor{Brown}{rgb}{0.45,0.0,0.05}
\usepackage{exscale,relsize}
\usepackage{amsmath}
\usepackage{amsfonts}
\usepackage{amssymb}
\usepackage{calc}
\usepackage{theorem}
\usepackage{pifont}      
\usepackage{graphicx}
\newcommand{\scal}[2]{\langle{{#1},{#2}}\rangle}

\newcommand{\exi}{\ensuremath{\exists\,}}

\newcommand{\RR}{\ensuremath{\mathbb R}}

\newcommand{\RX}{\ensuremath{\,\left]-\infty,+\infty\right]}}
\newcommand{\RXX}{\ensuremath{\,\left[-\infty,+\infty\right]}}

\newcommand{\NN}{\ensuremath{\mathbb N}}
\newcommand{\nnn}{\ensuremath{{n \in \NN}}}
\newcommand{\thalb}{\ensuremath{\tfrac{1}{2}}}

\newcommand{\menge}[2]{\big\{{#1} \mid {#2}\big\}}

\newcommand{\To}{\ensuremath{\rightrightarrows}}

\newcommand{\dom}{\ensuremath{\operatorname{dom}}}

\newcommand{\gra}{\ensuremath{\operatorname{gra}}}

\newcommand{\intdom}{\ensuremath{\operatorname{int}\operatorname{dom}}\,}
\newcommand{\inte}{\ensuremath{\operatorname{int}}}

\newcommand{\ran}{\ensuremath{\operatorname{ran}}}

\newcommand{\Id}{\ensuremath{\operatorname{Id}}}

\newcommand{\weakly}{\ensuremath{\,\rightharpoonup}\,}

\newcommand{\minf}{\ensuremath{-\infty}}
\newcommand{\pinf}{\ensuremath{+\infty}}

\renewcommand{\phi}{\ensuremath{\varphi}}

\newtheorem{theorem}{Theorem}[section]
\newtheorem{lemma}[theorem]{Lemma}
\newtheorem{fact}[theorem]{Fact}
\newtheorem{corollary}[theorem]{Corollary}
\newtheorem{proposition}[theorem]{Proposition}
\newtheorem{definition}[theorem]{Definition}

\theoremstyle{plain}{\theorembodyfont{\rmfamily}
}
\theoremstyle{plain}{\theorembodyfont{\rmfamily}
}
\theoremstyle{plain}{\theorembodyfont{\rmfamily}
}
\theoremstyle{plain}{\theorembodyfont{\rmfamily}
\newtheorem{example}[theorem]{Example}}
\theoremstyle{plain}{\theorembodyfont{\rmfamily}
\newtheorem{remark}[theorem]{Remark}}

\theoremstyle{plain}{\theorembodyfont{\rmfamily}
}

\begin{document}


\title{\sffamily{Rectangularity and paramonotonicity\\
of maximally monotone operators}}

\author{
Heinz H.\ Bauschke\thanks{Mathematics, Irving K.\ Barber School,
University of British Columbia, Kelowna, B.C. V1V 1V7, Canada. E-mail:
\texttt{heinz.bauschke@ubc.ca}.},\;
 Xianfu
Wang\thanks{Mathematics, Irving K.\ Barber School, University of British Columbia,
Kelowna, B.C. V1V 1V7, Canada. E-mail:
\texttt{shawn.wang@ubc.ca}.},\; and Liangjin\
Yao\thanks{Mathematics, Irving K.\ Barber School, University of British Columbia,
Kelowna, B.C. V1V 1V7, Canada.
E-mail:  \texttt{ljinyao@interchange.ubc.ca}.}}

\date{January 20, 2012}
\maketitle

\begin{abstract} \noindent
Maximally monotone operators play a key role in modern optimization and
variational analysis. Two useful subclasses are rectangular (also
known as star monotone) and paramonotone operators, which were
introduced by Brezis and Haraux, and by Censor, Iusem and Zenios,
respectively. The former class has useful range properties while
the latter class is of importance for interior point methods and
duality theory. 
Both notions are automatic for subdifferential operators and
known to coincide for certain matrices; however, more precise
relationships between rectangularity and paramonotonicity were
not known. 

Our aim is to provide new results and examples concerning these
notions. It is shown that rectangularity and paramonotonicity are
actually independent. 
Moreover, for linear relations, rectangularity implies
paramonotonicity but the converse implication requires additional
assumptions. We also consider continuous linear monotone operators,
and we point out that in Hilbert space 
both notions are automatic for certain displacement mappings. 
\end{abstract}

\noindent {\bfseries 2010 Mathematics Subject Classification:}\\
{Primary  47H05;
Secondary
47A06, 47B25, 47H09}

\noindent {\bfseries Keywords:}
Adjoint,
displacement mapping, 
linear relation,
maximally monotone,
nonexpansive,
paramonotone,
rectangular,
star monotone.

\section{Introduction}

Monotone operators continue to play a fundamental role
in optimization and nonlinear analysis as the
books 
\cite{BC2011},
\cite{BorVan},
\cite{Brezis}, 
\cite{BurIus},
\cite{GK},
\cite{GR}, 
\cite{ph},
\cite{Si},
\cite{Si2},
\cite{RockWets},
\cite{Zalinescu}, 
\cite{Zeidler2A}, and
\cite{Zeidler2B}
clearly demonstrate.
(See also the very recent thesis \cite{YaoPhD}.)
Let us start by reminding the reader on the various notions
and some key results.
To this end, we assume throughout the paper that
\begin{equation}
\text{
$X$ is a real Banach space with norm $\|\cdot\|$,
}
\end{equation}
that $X^*$ is the (continuous) dual of $X$, and
that $X$ and $X^*$ are paired by $\scal{\cdot}{\cdot}$.
It will be convenient at times to 
identify $X$ with its canonical image in the bidual space $X^{**}$.
Furthermore, $X\times X^*$ and $(X\times X^*)^* = X^*\times X^{**}$
are  paired via 
$\langle(x,x^*),(y^*,y^{**})\rangle=
\scal{x}{y^*} + \scal{x^*}{y^{**}}$, for all 
$(x,x^*)\in X\times X^*$ and  $(y^*,y^{**}) \in X^*\times X^{**}$.
Now let 
\begin{equation}
A\colon X\To X^*
\end{equation}
be a \emph{set-valued operator} from $X$ to $X^*$, 
i.e., $(\forall x\in X)$ $Ax\subseteq X^*$. 
We write $\gra A = \menge{(x,x^*)\in X\times X^*}{x^*\in Ax}$  for
the \emph{graph} of $A$. 
The \emph{inverse} of $A$, written $A^{-1}$  is given by
$\gra A^{-1} = \menge{(x^*,x)\in X^*\times X}{x^*\in Ax}$. 
The \emph{domain} of $A$ 
is $\dom A= \menge{x\in X}{Ax\neq\varnothing}$, while 
$\ran A=A(X)=\bigcup_{x\in X}Ax$ is 
the \emph{range} of $A$, and
$\ker A=\{x\mid 0\in Ax\}$ is 
the \emph{kernel}. 
Then $A$ is said to be \emph{monotone}, if 
\begin{equation}
\label{e:defmono}
\big(\forall (x,x^*)\in\gra A\big)
\big(\forall (y,y^*)\in\gra A\big)
\quad \scal{x-y}{x^*-y^*}\geq 0;
\end{equation}
if $A$ is monotone and it is impossible to properly enlarge $A$
(in the sense of graph inclusion), then $A$ is \emph{maximally monotone}. 
If the inequality in \eqref{e:defmono} is strict whenever $x\neq y$,
then $A$ is \emph{strictly monotone}. 
If $(z,z^*)\in X\times X^*$
is a point such that the operator with graph 
$\{(z,z^*)\}\cup \gra A$ is monotone,
then $(z,z^*)$ is \emph{monotonically related} to $\gra A$. 

The prime example of a maximally monotone operator is the subdifferential
operator of a function on $X$ that is convex, lower semicontinuous and
proper. However, not every maximally monotone operator arises in this
fashion (e.g., consider a rotator in the Euclidean plane). 
The books listed above have many results on maximally monotone operators.
There are two subclasses of monotone operators that are important
especially in optimization, namely rectangular 
(originally and also known as star or $3^*$ monotone operators) 
and paramonotone operators, which were
introduced by Brezis and Haraux (see \cite{BrezisHaraux}), 
and by Censor, Iusem and Zenios (see \cite{censor} and \cite{iusem}), 
respectively. 
This is due to the fact that the former class has very good range 
properties while the latter class is important in the
study of interior point methods for variational inequalities.
(For very recent papers in which these notions play a central role,
we refer the reader to \cite{BBHM}, \cite{BMMW}, and \cite{BMWrange}.)
Before we turn to precise definitions of these notion, we
recall key properties of the Fitzpatrick function, which has proven
to be a crucial tool in the study of
(maximally) monotone operators. 

\begin{fact}[Fitzpatrick]
{\rm (See {\cite[Proposition~3.2 and Corollary~3.9]{Fitz88}}.)}
\label{f:Fitz} 
Let $A\colon X\To X^*$
 be monotone such that $\gra A\neq\varnothing$.
Define the \emph{Fitzpatrick function} associated with $A$ 
by 
\begin{equation*}
F_A\colon X\times X^*\to\RX\colon
(x,x^*)\mapsto \sup_{(a,a^*)\in\gra A}
\big(\scal{x}{a^*}+\scal{a}{x^*}-\scal{a}{a^*}\big).
\end{equation*}
Then $F_A$ is proper, lower semicontinuous and convex, and
 $F_A=\langle \cdot,\cdot\rangle$ on $\gra A$.
If $A$ is maximally monotone,
then $\scal{\cdot}{\cdot}\leq F_A$ with equality precisely on $\gra A$. 
\end{fact}

We now present the formal definitions of rectangularity and
paramonotonicity.

\begin{definition}[rectangular and paramonotone]
Let $A\colon X\To X$ be a monotone operator such that 
$\gra A\neq\varnothing$.
Then:
\begin{enumerate}
\item
$A$ is \emph{rectangular
(which is also known as $*$ or $3^*$ monotone; 
see \cite{BrezisHaraux}, \cite[Definition~31.5]{Si2},
and \cite[Definition~32.40(c) on page~901]{Zeidler2B})} if
\begin{align}
\dom A\times \ran A\subseteq\dom F_A.
\end{align}
\item 
$A$ is \emph{paramonotone (see \cite{censor} and \cite{iusem})}
if the implication
\begin{equation}
\left.
\begin{array}{c}
(x,x^*)\in\gra A\\
(y,y^*)\in\gra A\\
\langle x-y,x^*-y^*\rangle=0
\end{array}\right\}
\;\;\Rightarrow\;\;
\big\{(x,y^*), (y,x^*)\big\}\subseteq\gra A
\end{equation}
 holds.
\end{enumerate}
\end{definition}

The following two results illustrate that
rectangularity and paramonotonicity are automatic
for subdifferential operators from convex analysis. 
The first result is due to Brezis and Haraux, who considered
the Hilbert space case in \cite{BrezisHaraux} 
(see \cite[Proposition~32.42]{Zeidler2B} for the Banach space version).

\begin{fact}[Brezis-Haraux]
\label{f:B-H} 
Let $f\colon X\to \RX$ be convex, lower semicontinuous, and proper.
Then $\partial f$ is rectangular.
\end{fact}

The second fact---in a finite-dimensional setting---is 
due to Censor, Iusem and Zenios (see \cite{censor} and \cite{iusem},
as well as \cite{BWY8} 
for an extension to Banach space with a different proof).

\begin{fact}[Censor-Iusem-Zenios]\label{paramono}
Let $f\colon X\to \RX$ be convex, lower semicontinuous, and proper.
Then $\partial f$ is paramonotone.
\end{fact}

Besides these results for subdifferential operators,
it was known that rectangularity and paramonotonicity coincide
for certain matrices (see \cite[Remark~4.4]{BBW}). Thus, previously,
it was not clear whether these notions are different.

\emph{The aim of this work is to systematically study
rectangular and paramonotone operators.}

Let us summarize our key findings. 

\begin{itemize}
\item Rectangularity and paramonotonicity are \emph{independent}
notions. 
\item For linear relations, \emph{rectangularity implies
paramonotonicity but not vice versa}.
\item For linear relations satisfying certain closure assumptions in
reflexive spaces, 
\emph{rectangularity and paramonotonicity coincide and also hold for
the adjoint}.
\item For displacement mappings of nonexpansive operators in Hilbert
space, \emph{rectangularity and paramonotonicity are automatic}. 
\end{itemize}

The remainder of this paper is organized as follows.
In Section~\ref{s:aux}, we collect auxiliary results for the reader's
convenience and future use. General monotone operators
are the topic of Section~\ref{s:general} where we also present
an example (see Example~\ref{ExaRniP}) 
illustrating that rectangularity does not imply
paramonotonicity. 
In Section~\ref{s:main}, we focus on linear relations.
For this important subclass of monotone operators,
we obtain characterizations and relationships to corresponding
properties of the adjoint. It is true that 
rectangularity implies paramonotonicity in this setting
(see Proposition~\ref{PropRDP:1}).
Our main result (Theorem~\ref{PaFit:1}) presents a pleasing
characterization in the reflexive setting under a mild closedness
assumption. 
The setting of continuous linear monotone operators is considered
in Section~\ref{s:linop} where we present characterizations of
rectangularity and an example of a paramonotone operator that is not
rectangular (Example~\ref{ExNR:1}). 
In the final Section~\ref{s:id-t}, we consider monotone operators
that are displacement mappings of nonexpansive operators in Hilbert
space. 
For this class of operators, rectangularity and paramonotonicity are
automatic. 

For the convenience of the reader, let us recall some (mostly
standard) terminology from convex analysis and the theory of linear
relations. 

Given a subset $C$ of $X$,
$\inte C$ and $\overline{C}$ denote 
the \emph{interior} and \emph{closure} of $C$, respectively. 
For every $x\in
X$, the \emph{normal cone} of $C$ at $x$ is defined by
$N_C(x)= \menge{x^*\in X^*}{\sup_{c\in C}\scal{c-x}{x^*}\leq 0}$, if
$x\in C$; and $N_C(x)=\varnothing$, if $x\notin C$.
Let $f\colon X\to \RXX$. Then
$\dom f= \menge{x\in X}{f(x)<\pinf})$ 
is the (essential) \emph{domain} of $f$, and
$f^*\colon X^*\to\RXX\colon x^*\mapsto
\sup_{x\in X}(\scal{x}{x^*}-f(x))$ is
the \emph{Fenchel conjugate} of $f$. 
We say that $f$ is \emph{proper} 
if $\dom f\neq\varnothing$ and $\minf\notin\ran f$.
Let $f$ be proper.
The \emph{subdifferential operator} of $f$ 
$\partial f\colon X\To X^*\colon
x\mapsto \menge{x^*\in X^*}{(\forall y\in
X)\; \scal{y-x}{x^*} + f(x)\leq f(y)}$.

Let $Z$ be  a real  Banach space with continuous dual $Z^*$, 
let $C \subseteq Z$ and $D\subseteq Z^*$.
We write $C^\perp := \menge{z^*\in Z^*}{(\forall c\in C)\;\; \langle
z^*,c\rangle= 0}$ and $D_{\perp} := D^\perp \cap Z$. 

Now let $A\colon X\To X^*$ be a 
\emph{linear relation}, i.e., $\gra A$ is a linear subspace of $X$.
(See \cite{Cross} for background material on linear relations.)
The \emph{adjoint} of $A$,
written as $A^*$, is defined by
\begin{subequations}
\begin{align}
\gra A^* &=
\big\{(x^{**},x^*)\in X^{**}\times X^*\mid(x^*,-x^{**})\in(\gra A)^{\bot}\big\}\\
&=\big\{(x^{**},x^*)\in X^{**}\times X^*\mid
(\forall(a,a^*)\in \gra A)\;
\langle x^*,a\rangle=\langle a^*,x^{**}\rangle \big\}.
\end{align}
\end{subequations}
Furthermore, $A$ is \emph{symmetric} if $\gra A\subseteq \gra A^*$;
similarly, $A$ is \emph{skew} if $\gra A \subseteq \gra(-A^*)$. 
The \emph{symmetric part} of $A$ is defined by $A_+ = \thalb A + \thalb
A^*$, while the \emph{skew part} of $A$ is $A_{\mathlarger{\circ}} = \thalb A-\thalb A^*$.
We denote by $\Id$ the \emph{identity mapping} on $X$.
Finally, the {closed unit ball} in $X$ is denoted by $B_X=
\menge{x\in X}{\|x\|\leq1}$, and the positive integers by 
$\NN=\{1,2,3,\ldots\}$.

 \section{Auxiliary results}\label{s:aux}

We start with an elementary observation that has turned out to be quite
useful in \cite{PheSim}. 

\begin{fact}
\label{FQ:1}
Let $\alpha$, $\beta$, and $\gamma$ be in $\RR$. 
Then 
\begin{equation}
(\forall t\in\RR)\quad 
\alpha t^2+\beta t+\gamma \geq 0
\end{equation}
if and only if
$\alpha \geq 0$, $\gamma \geq 0$,
and $\beta^2\leq 4\alpha\gamma$.
\end{fact}

Next is a result from Convex Analysis that will be needed in the sequel.

\begin{fact}\emph{(See \cite[Proposition~3.3 and Proposition~1.11]{ph}.)}
\label{pheps:1}Let $f:X\rightarrow\RX$ be lower semicontinuous and
convex, with $\intdom f\neq\varnothing$.
Then $f$ is continuous on $\intdom f$ and $\partial f(x)\neq\varnothing$ for every $x\in\intdom f$.
\end{fact}

We now state a very useful sufficient condition for maximal
monotonicity and continuity.  

\begin{fact}[Phelps-Simons]\emph{(See 
\cite[Corollary 2.6 and Proposition~3.2(h)]{PheSim}.)}\label{F:1}
Let $A\colon X\rightarrow X^*$ be  monotone and linear. Then $A$ is
maximally monotone and continuous.
\end{fact}

Most of our results involve linear relations.
The next result holds even without monotonicity. 

\begin{fact}[Cross]\label{Rea:1}
Let $A:X \To X^*$ be a linear relation.
Then the following hold:
\begin{enumerate}
\item \label{Th:28}
$(\forall (x,x^*)\in\gra A)$ 
$Ax=x^* +A0$. 
\item \label{Sia:2b}
$(\forall x\in \dom A)(\forall y\in\dom A^*)$
$\langle A^*y,x\rangle=\langle y,Ax\rangle$ is a singleton.
\item\label{Th:32s} 
$\ker A^*=(\ran A)^{\bot}$ and $\ker\bar{A}=(\ran A^*)_{\bot}$,
 where $\gra \bar{A} = \overline{\gra A}$. 
\end{enumerate}
\end{fact}
\begin{proof}
\ref{Th:28}: See \cite[Proposition~I.2.8(a)]{Cross}.
\ref{Sia:2b}: See \cite[Proposition~III.1.2]{Cross}.
\ref{Th:32s}: See \cite[Proposition~III.1.4(a)\&(c)]{Cross}.
\end{proof}

The following result gives a  sufficient condition for maximal
monotonicity. 

\begin{fact}{\rm (See \cite[Theorem~3.1]{Yao2}.)}\label{domain:L1}
Let $A:X\To X^*$ be a maximally monotone linear relation, 
and let $f:X\to\RX$ be convex, lower semicontinuous, and convex with 
$\dom  A\cap\intdom\partial f\neq\varnothing$.
Then $A+\partial f$ is maximally monotone.
\end{fact}

For a monotone linear relation $A\colon X \To X^*$ it will be convenient
to define---as in, e.g., \cite{BBW}---the associated quadratic form
\begin{equation}
q_A\colon X\to \RX\colon x\mapsto  \begin{cases} \tfrac{1}{2}\langle x,Ax\rangle,
&\text{if $x\in \dom A$};\\
+\infty,&\text{otherwise}.\end{cases}
\end{equation}

Let us record some useful 
properties of monotone linear relations for future reference.

\begin{fact}
{\rm (See \cite[Proposition~5.2(i)(iii)\&(iv), 
Proposition~5.3 and Proposition~5.4(iii)\&(iv)]{BBWY5}.)}\label{linear}
Let $A\colon X \To X^*$ be a maximally monotone linear relation.
Then the following hold:
\begin{enumerate}
\item\label{sia:2v}  $A0=A^*0=A_+0 = (\dom A)^{\perp}$ is
(weak$^*$) closed.
\item \label{sia:3vi}
If  $\dom A$ is closed, then $X\cap \dom A^*=\dom A$.
\item \label{Nov:s2}$q_A$ is single-valued and convex.
\item \label{f:PheSim:brah:1} $A^*|_X$ is monotone.
\item \label{dms:quad}
If  $\dom A$ is closed,
then $A_+=\partial \overline{q_A}$, where
$\overline{q_A}$ is the lower semicontinuous hull of $q_A$.
\end{enumerate}
\end{fact}

Concerning kernels and ranges, we have the following result. 

\begin{fact}\emph{(See \cite[Theorem~3.2(i)\&(ii)]{BWY3}.)}\label{KerRa:1}
Suppose that $X$ is reflexive, and let $A:X\To X^*$ be  a
maximally monotone linear relation.
Then $\ker A=\ker A^*$ and $\overline{\ran A}=\overline{\ran A^*}$.
\end{fact}

The next result provides a formula 
connecting the Fitzpatrick function
of a given linear relation with that of its adjoint.

\begin{lemma}\label{LemmaCP:3a}
Let $A\colon X\To  X^*$ be a linear relation such that 
$\dom A=\dom (A^*|_X)$, and let $z^*\in X^*$.
Then 
\begin{subequations}
\begin{align}
\big(\forall (x,x^*)\in\gra A\label{LCP:3ae1}\big)
\quad F_{A^*|_X}(x,z^*)&=F_A(0,x^*+z^*)\\
\big(\forall (x,y^*)\in\gra (A^*|_X)\label{LCP:3ae2}\big)\quad
F_A(x,z^*)&=F_{A^*|_X}(0, y^*+z^*). 
\end{align}
\end{subequations}
\end{lemma}
\begin{proof}
 Let $(x,x^*)\in\gra A$. Then
 \begin{subequations}
\begin{align}
F_{A^*|_X}(x,z^*)&=\sup_{(a,b^*)\in \gra (A^*|_X)}
\big(\langle x, b^*\rangle+\langle a,z^*\rangle-\langle a,
b^*\rangle\big)\\
&=\sup_{(a,a^*)\in \gra A}\big(\langle x^*, a\rangle+
\langle a,z^*\rangle-\langle a, a^*\rangle\big)
\quad\text{(by Fact~\ref{Rea:1}\ref{Sia:2b})}\\
&=F_A(0,x^*+z^*). 
\end{align}
\end{subequations}
This establishes \eqref{LCP:3ae1}.
The proof of \eqref{LCP:3ae2} is similar. 
\end{proof}

Finally, ``self-orthogonal'' elements of the graph of $-A$ must
belong to the graph of $A^*$. 

\begin{lemma}
\label{LemmaCP:}
Let $A\colon X\To  X^*$ be a  monotone linear relation and 
suppose that $(a,a^*)\in \gra A$ satisfies
$\scal{a}{a^*}=0$. Then $(a,-a^*)\in\gra A^*$.
\end{lemma}
\begin{proof}
Take $(b,b^*)\in\gra A$ and set 
\begin{align}
\label{PaAd:6}
\delta := \langle a, b^*\rangle+\langle b,a^*\rangle.
\end{align}
The monotonicity of $A$ and linearity of $\gra A$ yield
\begin{equation}
(\forall t\in\RR)\quad 
t^2\langle b,b^*\rangle-t\delta=\langle a-tb, a^*-tb^*\rangle\geq 0.
\end{equation}
By Fact~\ref{FQ:1},
$\delta^2\leq0$ and thus $\delta=0$.
Therefore, $(a^*,a)\in (\gra A)^{\bot}$ by \eqref{PaAd:6}.
Hence it follows that $(a,-a^*)\in\gra A^*$.
\end{proof}

\section{General results}

\label{s:general}

The results in this section pertain to general operators.
We start with a simple observation concerning the sum of two paramonotone
operators. 

\begin{proposition}\label{PropRDP:2}
Let $A$ and $B$ be paramonotone operators on $X$. 
Then $A+B$ is paramonotone.
\end{proposition}
\begin{proof}
Suppose that 
$\{(x_1,a_1^*),(x_2,a_2^*)\}\subseteq\gra A$ and
$\{(x_1,b_1^*),(x_2,b_2^*)\}\subseteq\gra B$ are such that
$\langle  (a_1^*+b_1^*)-(a_2^*+b_2^*), x_1-x_2\rangle=0$.
Then
$\langle  a_1^*-a_2^*, x_1-x_2\rangle+\langle  b_1^*-b_2^*,
x_1-x_2\rangle=0$. 
Combining with the monotonicity of $A$ and $B$, we see that 
\begin{equation}
\langle  a_1^*-a_2^*, x_1-x_2\rangle=0\quad\text{and}
\quad\langle  b_1^*-b_2^*, x_1-x_2\rangle=0.\label{Prasum:2}
\end{equation}
Since $A$ and $B$ are paramonotone and by \eqref{Prasum:2}, we have
$\{(x_1,a_2^*),(x_2,a_1^*)\}\subseteq\gra A$
and $\{(x_1,b_2^*),(x_2,b_1^*)\}\subseteq\gra B$. 
Thus
$\{(x_1,a_2^*+b^*_2),(x_2,a_1^*+b^*_1)\}\subseteq\gra (A+B)$.
Therefore, $A+B$ is paramonotone.
\end{proof}

The following result, which provides a useful sufficient condition for
rectangularity, was first proved by Brezis and Haraux in 
\cite[Example~2]{BrezisHaraux} in a Hilbert space setting. 
In fact, their result holds in general Banach space.
For completeness, we include the proof. 

\begin{proposition}\label{RecT:1}
Let $A:X\To X^*$ be monotone such that $\gra A\neq\varnothing$.
Suppose that $A$ is \emph{strongly coercive} in the sense that 
\begin{equation}
(\forall x\in\dom A)\quad 
\lim_{\rho\to\pinf} \;\;
\inf_{(a,a^*)\in\gra A \text{\rm ~and~} \|a\|\geq\rho}\;\;
\frac{\langle a^*,a-x\rangle}{\|a\|}=\pinf. 
\end{equation} 
Then $\dom A\times X^*\subseteq\dom F_A$;
consequently, $A$ is rectangular.
\end{proposition}
\begin{proof}
Let $(x,y^*)\in\dom A\times X^*$, and set $M:=\|y^*\|+1$. 
Then there exists $\rho>0$ such that
for every $(a,a^*)\in\gra A$ with $\|a\|\geq \rho$, we have
$\langle a^*,a-x\rangle\geq M\|a\|$.
Thus, 
\begin{equation}
\label{sucore:1}
\big(\forall (a,a^*)\in \gra A\big)\quad
\|a\|\geq \rho \;\;\Rightarrow\;\;
\langle a^*,x-a\rangle\leq -M\|a\|.
\end{equation}
Take $x^*\in Ax$. Then, by monotonicity of $A$, 
\begin{equation}
\label{sucore:2}
\big(\forall (a,a^*)\in \gra A\big)\quad
\langle a^*, x-a\rangle\leq\langle  x^*,x-a\rangle.
\end{equation}
Let us now evaluate $F_A(x,y^*)$. 
Fix $(a,a^*)\in\gra A$. 

\noindent
\texttt{Case 1:} $\|a\|<\rho$.\\
Using \eqref{sucore:2}, we estimate 
$\scal{a^*}{x-a}+\scal{a}{y^*} \leq 
\scal{x^*}{x-a}+\scal{a}{y^*} \leq 
\|x^*\|\cdot\|x-a\|+\|a\|\cdot\|y^*\| \leq 
\|x^*\|(\|x\|+\|a\|)+\|a\|\cdot\|y^*\|
\leq \|x^*\|(\|x\|+\rho)+\rho\|y^*\|$. 

\noindent
\texttt{Case 2:} $\|a\|\geq \rho$.\\
Using \eqref{sucore:1}, we have 
$\scal{a^*}{x-a}+\scal{a}{y^*} \leq 
-M\|a\|+\|a\|\cdot\|y^*\| = \|a\|\cdot(\|y^*\|-M) = -\|a\|\leq 0$. 

Altogether, we conclude that 
\begin{equation}
F_A(x,y^*) = \sup_{(a,a^*)\in\gra A}
\scal{a^*}{x-a}+\scal{a}{y^*} \leq 
\|x^*\|\big(\|x\|+\rho\big)+\rho\|y^*\|.
\end{equation}
Hence $\dom A\times X^*\subseteq\dom F_A$ and 
$A$ is therefore rectangular.
\end{proof}

\begin{example}
\label{ex:boundeddom}
Let $A\colon X\To X^*$ be monotone such that $\dom A$ is nonempty and
bounded. Then $A$ is rectangular.
\end{example}
\begin{proof}
This is immediate from Proposition~\ref{RecT:1} because
$\inf\varnothing=\pinf$ and therefore $A$ is strongly coercive.
\end{proof}

We now show how to construct a
maximally monotone operator that is rectangular but not paramonotone.

\begin{proposition}\label{PRcapa}
Let $A:X\To X^*$ be a maximally monotone linear relation,
and let $C$ be a bounded closed convex subset of $X$
such that $0\in\inte C$.  
Then $A+N_C$ is maximally monotone and rectangular.
If $A$ is not paramonotone, then neither is $A+N_C$. 
\end{proposition}
\begin{proof}
Set $B = A+N_C$. 
By Fact~\ref{domain:L1} or \cite[Theorem~3.1]{BWY4}, 
$B$ is maximally monotone. 
Since $\dom B = \dom A \cap C \subseteq C$ is bounded, 
we deduce from Example~\ref{ex:boundeddom} that $B$ is rectangular. 

Now assume in addition that $A$ is not paramonotone. 
In view of Lemma~\ref{LemmaCP:1} below,
there exists $(a,a^*)\in\gra A$ such that 
\begin{equation}
\label{Rcapa:1}
\langle a,a^*\rangle=0 \quad\text{but}\quad a^*\notin A0.
\end{equation}
Since $0\in\inte C$, there exists $\delta>0$ such that
$\delta a\in \inte C$. 
The linearity of $\gra A$ yields $(\delta a,\delta a^*)\in\gra A$.
Hence $\{(\delta a,\delta a^*), (0,0)\}\subseteq\gra B$, 
and $\langle \delta a-0,\delta a^*-0\rangle=0$ by \eqref{Rcapa:1}.
However, \eqref{Rcapa:1} implies that $\delta a^*\notin A0=B0$.
Therefore, $B$ is not paramonotone.
\end{proof}

We conclude this section with our first counterexample. 

\begin{example}[rectangular $\not\Rightarrow$ paramonotone in the general
case]
\label{ExaRniP}
Suppose that $X=\RR^2$ and set 
\begin{equation}
A:=\begin{bmatrix}
0& 1\\
-1& 0
\end{bmatrix}. 
\end{equation} 
Then  $A+N_{B_X}$ is
maximally monotone and rectangular, but not paramonotone.
\end{example}
\begin{proof}
Clearly, $0\in\inte B_X$.
The lack of paramonotonicity of $A$ is a consequence of Lemma~\ref{LemmaCP:1}
below. 
The conclusion therefore follows from Proposition~\ref{PRcapa}. 
\end{proof}

\section{Linear relations}\label{s:main}

In this section, we focus exclusively on linear relations.
We start with characterizations of rectangularity and paramonotonicity.
These yield information about corresponding properties of the adjoint.

\begin{lemma}[characterization of rectangularity]
\label{lemR:1}
Let $A:X\To X^*$ be a monotone linear relation.
Then $A$ is rectangular $\Leftrightarrow$
$\dom A\times \{0\}\subseteq\dom F_A$.
\end{lemma}

\begin{proof}
``$\Rightarrow$": Clear.
``$\Leftarrow$": Let $x\in\dom A$ and $(y,y^*)\in \gra A$. 
Then 
\begin{subequations}
\begin{align}
F_A(x,y^*)&= F_A\big(\tfrac{1}{2}(2y, 2y^*)+\tfrac{1}{2}(2x-2y,0)\big)\\
&\leq\tfrac{1}{2}F_A(2y,2y^*)+\tfrac{1}{2}F_A(2x-2y,0)\quad
\text{(since $F_A$ is convex by Fact~\ref{f:Fitz} )}\\
&=\tfrac{1}{2}\langle 2y,2y^*\rangle+\tfrac{1}{2}F_A(2x-2y,0)
\quad\text{(by Fact~\ref{f:Fitz} and since $\gra A$ is linear)}\\
&<+\infty\quad\text{(since $(2x-2y,0)\in\dom F_A$)}.
\end{align}
\end{subequations}
Thus $(x,y^*)\in\dom F_A$ and hence $A$ is rectangular.
\end{proof}

\begin{corollary}[rectangularity of the adjoint]
\label{LemmaCP:3}
Let $A\colon X\To  X^*$ be a maximally monotone linear relation such
that $\dom A$ is closed.
Then $A$ is rectangular if and only if
 $A^*|_X$ is rectangular.
\end{corollary}
\begin{proof}
Observe first that Fact~\ref{linear}\ref{sia:3vi} yields
$\dom (A^*|_X)=\dom A$.
``$\Rightarrow$": Let $x\in \dom (A^*|_X)$ and $x^*\in Ax$.  
Applying Lemma~\ref{LemmaCP:3a} and the rectangularity of $A$,
we obtain 
\begin{equation}
F_{A^*|_X}(x,0)
=F_A(0,x^*)<\pinf.
\end{equation}
Combining with Fact~\ref{linear}\ref{f:PheSim:brah:1} and Lemma~\ref{lemR:1},
we deduce that 
 $A^*|_X$ is rectangular.
``$\Leftarrow$'': The proof is similar to the just established
implication and thus omitted. 
\end{proof}

\begin{lemma}[characterization of paramonotonicity]
\label{LemmaCP:1}
Let $A:X\rightrightarrows X^*$ be a monotone linear relation.
Then $A$ is paramonotone if and only if the implication
\begin{equation}
\left.\begin{array}{c}
(a,a^*)\in\gra A\\ \langle a,a^*\rangle=0
\end{array}
\right\} 
\;\; \Rightarrow \;\;
a^*\in A0
\end{equation}
 holds.
\end{lemma}
\begin{proof}
``$\Rightarrow$'': Clear.
``$\Leftarrow$'': Let $(a,a^*)$ and $(b,b^*)$ be in $\gra A$ be such that
\begin{equation}
\label{LemmP:1}
\langle a-b,a^*-b^*\rangle=0.
\end{equation}
Since $\gra A$ is linear, it follows that $(a-b,a^*-b^*)\in\gra A$.
The hypothesis now implies that $a^*-b^*\in A0$. 
Thus, by Fact~\ref{Rea:1}\ref{Th:28}, we have $Aa=Ab$. 
Hence $a^*\in Ab$ and $b^*\in Aa$.
Therefore, $A$ is paramonotone.
\end{proof}

\begin{proposition}[paramonotonicity of the adjoint]
\label{LemmaCP:2}
Let $A\colon X\To  X^*$ be a maximally monotone linear relation.
Then $A$ is paramonotone if and only if $A^*|_X$ is paramonotone.
\end{proposition}
\begin{proof}
We start by noting that $A^*|_X$ is monotone by 
Fact~\ref{linear}\ref{f:PheSim:brah:1}. 

``$\Rightarrow$'': 
 Let $(a,a^*)\in\gra (A^*|_X)$ be such that
\begin{align}
\langle a,a^*\rangle=0.\label{PaAd:1}
\end{align}
One verifies that $(-a,a^*)$ is monotonically related to $\gra A$.
The maximal monotonicity of $A$ yields $(-a,a^*)\in\gra A$. 
Hence, by \eqref{PaAd:1} and since $A$ is paramonotone, $a^*\in A0$. 
Thus, by Fact~\ref{linear}\ref{sia:2v}, $a^*\in A^*0=(A^*|_X) 0$.
It now follows from Lemma~\ref{LemmaCP:1} that $A^*|_X$ is paramonotone.

``$\Leftarrow$'':  
Let $(a,a^*)\in\gra A$ be such that
$\langle a,a^*\rangle=0$.
By Lemma~\ref{LemmaCP:}, $(a,-a^*)\in\gra (A^*|_X)$.
Since $A^*|_X$ is paramonotone and  using Fact~\ref{linear}\ref{sia:2v},
we deduce that 
$a^*\in (A^*|_X) 0=A^*0=A0$.
Lemma~\ref{LemmaCP:1} now implies that $A$ is paramonotone.
\end{proof}

The next result shows that rectangularity is a sufficient condition for 
paramonotonicity in the linear case investigated in this section.

\begin{proposition}[rectangular $\Rightarrow$ paramonotone in the linear
case]
\label{PropRDP:1}
Let $A:X\To X^*$ be a maximally monotone linear relation such that
$A$ is rectangular. Then $A$ is paramonotone.
\end{proposition}
\begin{proof}
Let $(a,a^*)\in\gra A$ be such that
$\langle a,a^*\rangle=0$.
Take $b\in\dom A$. Then, since $A$ is linear and rectangular, 
\begin{equation}
\pinf> F_A(b,0)\geq\sup_{t\in\RR}
\big(\langle b, ta^*\rangle-\langle ta,ta^*\rangle\big)
=\sup_{t\in\RR}
\langle b, ta^*\rangle.
\end{equation}
Hence $a^*\in (\dom A)^\perp$. 
In view of Fact~\ref{linear}\ref{sia:2v}, $a^*\in A0$. 
Therefore, by Lemma~\ref{LemmaCP:1}, $A$ is paramonotone.
\end{proof}

\begin{remark} 
Some comments regarding Proposition~\ref{PropRDP:1} are in order.
\begin{enumerate}
\item 
The linearity assumption on $A$ in Proposition~\ref{PropRDP:1} is 
not superfluous, see 
Example~\ref{ExaRniP} above. 
\item The converse implication in Proposition~\ref{PropRDP:1} fails even
when $A$ is additionally assumed to be single-valued and continuous; 
see Example~\ref{ExNR:1} below. 
\end{enumerate}
\end{remark}

Proposition~\ref{PropRDP:1} raises the question when paramonotonicity
implies rectangularity. Our main result which we state next shows
that this implication holds under relatively 
mild assumptions.

\begin{theorem}[main result] \label{PaFit:1}
Let $A: X\To  X^*$ be a maximally monotone linear relation.
Suppose that $X$ is reflexive, 
and that $\dom A$ and $\ran A_+$ are closed.
Then the following are equivalent: 
\begin{enumerate}
\item\label{RT:1} $A$ is rectangular.

\item \label{RT:2} $\overline{\ran A}=\ran  A_+$.

\item \label{RT:3} $\ker A_+=\ker A$.
\item\label{RT:4} $A$ is paramonotone.
\item\label{RT:5} $A^*$ is paramonotone.
\item\label{RT:6} $A^*$ is rectangular.
\end{enumerate}
\end{theorem}
\begin{proof} 
Note that Fact~\ref{linear}\ref{sia:3vi} yields
$\dom A=\dom A^*$. 
Thus, by Fact~\ref{Rea:1}\ref{Sia:2b}, we have
\begin{equation}
\label{parke:1}
(\forall x\in\dom A)\quad
\langle Ax,x\rangle=\langle A_+x,x\rangle.
\end{equation}
Now take $y\in\dom A$ such that $\langle Ay,y\rangle=0$.
By \eqref{parke:1}, $\langle A_+y,y\rangle=0$.
Hence, by Fact~\ref{linear}\ref{dms:quad} and Fact~\ref{paramono},
$0\in A_+y$ and thus $y\in\ker A_+$.
We have established the implication 
\begin{equation}
\label{ParKe:1}
y\in\dom A\text{~and~} \langle Ay,y\rangle=0
\quad \Rightarrow \quad y\in\ker A_+\,.
\end{equation}

``\ref{RT:1}$\Leftrightarrow$\ref{RT:2}'': Let $x\in\dom A$.
Then $x\in\dom A^*$ by Fact~\ref{linear}\ref{sia:3vi}.
Hence, by Fact~\ref{Rea:1}\ref{Sia:2b},
\begin{equation}
\label{parr:1}
F_A(x,0)=\tfrac{1}{2}q^*_A(A^*x).
\end{equation}
From the Br{\o}ndsted-Rockafellar Theorem
(see, e.g., \cite[Theorem~3.1.2]{Zalinescu}) and 
Fact~\ref{linear}\ref{dms:quad} it follows that 
$\ran A_+\subseteq\dom \overline{q_A}^*\subseteq\overline{\ran A_+}=\ran A_+$.
Thus $\ran A_+=\dom \overline{q_A}^*=\dom q^*_A$.
Hence, using \eqref{parr:1}, Lemma~\ref{lemR:1},
the closedness of $\ran A_+$, and Fact~\ref{KerRa:1},
we deduce that 
\begin{subequations}
\label{SuPR:2}
\begin{align}
\text{$A$ is rectangular} 
&\Leftrightarrow \ran A^*\subseteq \ran A_+\\
&\Leftrightarrow \overline{\ran A^*}\subseteq \ran A_+\\
&\Leftrightarrow \overline{\ran A}\subseteq
\ran A_+.
\end{align}
\end{subequations}
On the other hand,
by Fact~\ref{KerRa:1}, $\ran A_+\subseteq\ran A+\ran A^*\subseteq\ran A+\overline{\ran A}=\overline{\ran A}$.
Thus recalling \eqref{SuPR:2}, we see that 
$A$ is rectangular $\Leftrightarrow$
$\overline{\ran A}=\ran A_+$.

``\ref{RT:2}$\Leftrightarrow$\ref{RT:3}'': By Fact~\ref{Rea:1}\ref{Th:32s} and Fact~\ref{KerRa:1},
we have $(\ran A)^{\bot}=\ker A^*=\ker A$.
By Fact~\ref{linear}\ref{Nov:s2}\&\ref{dms:quad}, $A_+$ is maximally monotone 
and so $\gra A_+$ is closed.
By \cite[Proposition~2.8]{BWY8}, $(A_+)^*=A_+$.
Hence, using Fact~\ref{Rea:1}\ref{Th:32s} again, 
we obtain $(\ran A_+)^{\bot}=\ker A_+$. 
This establishes the equivalence \ref{RT:2}$\Leftrightarrow$\ref{RT:3}. 

``\ref{RT:3}$\Rightarrow$\ref{RT:4}'':
Let $(a,a^*)\in\gra A$ be such that
$\langle a,a^*\rangle=0$.
Then, by \eqref{ParKe:1} and the assumption, $a\in\ker A_+=\ker A$.
Thus, $0\in Aa$. By Fact~\ref{Rea:1}\ref{Th:28},
$Aa=0+A0=A0$ and hence $a^*\in A0$. 
Lemma~\ref{LemmaCP:1} now implies that $A$ is paramonotone.

``\ref{RT:4}$\Rightarrow$\ref{RT:3}'':
Let $x\in\ker A_+$ and $x^*\in Ax$.
Then by \eqref{parke:1}, $\langle x^*,x\rangle=\langle Ax,x\rangle=\langle A_+x,x\rangle=0$.
Since $A$ is paramonotone, $x^*\in A0$ and hence $Ax=x^*+A0=A0$.
Thus $x\in\ker A$ and so
\begin{equation}
\label{SuPR:3}
\ker A_+\subseteq\ker A.
\end{equation}
On the other hand, take $x\in \ker A$. Then
$\langle Ax,x\rangle=\scal{0}{x}=0$.
Thus, by \eqref{ParKe:1}, $x\in\ker A_+$.
We have shown that $\ker A\subseteq\ker A_+$.  
Combining with \eqref{SuPR:3}, we therefore obtain
$\ker A=\ker A_+$.

``\ref{RT:4}$\Leftrightarrow$\ref{RT:5}'': 
Proposition~\ref{LemmaCP:2}.
``\ref{RT:1}$\Leftrightarrow$\ref{RT:6}'': 
Corollary~\ref{LemmaCP:3}.
\end{proof}

\begin{remark}
In Theorem~\ref{PaFit:1}, the assumption 
that $\ran A_+$ is closed is not superfluous: 
indeed, let $A$ and $B$ be defined as in 
Example~\ref{ExNR:1} below,
where $A$ and $B$ are even continuous linear monotone operators defined on
a Hilbert space, 
and set $C = A+B$. 
Then 
\begin{equation}
\ran C_+=\ran\frac{A+B+(A+B)^*}{2}=\ran\frac{A+B+A-B}{2}=\ran A,
\end{equation}
is a proper dense subspace of $X$, and $C$ is paramonotone 
but not rectangular. Note that $\dom C=X$.
We do not know whether the assumption on the closure of the domain
in Theorem~\ref{PaFit:1} is superfluous. 
\end{remark}

The statement of Theorem~\ref{PaFit:1} simplifies significantly in the 
finite-dimensional setting that we state next.
(See also \cite[Proposition~3.2]{iusem} and \cite[Remark~4.11]{BBW}
for related results pertaining to the case when 
$A$ is single-valued and thus identified with a matrix.)

\begin{corollary}[finite-dimensional setting]
Suppose that $X$ is finite-dimensional,
and let $A\colon X\To  X^*$ be a maximally monotone linear relation.
Then the following are equivalent: 
\begin{enumerate}
\item\label{RTC:1} $A$ is rectangular.
\item \label{RTC:2} $\ran A_+=\ran  A$.
\item \label{RTC:3} $\ker A_+=\ker A$.
\item\label{RTC:4} $A$ is paramonotone.
\item\label{RTC:5} $A^*$ is paramonotone.
\item\label{RTC:6} $A^*$ is rectangular.
\end{enumerate}
\end{corollary}

\section{Linear (single-valued) operators}

\label{s:linop}

The following result, which is complementary to 
Lemma~\ref{LemmaCP:3} and provides characterizations of rectangularity for
continuous linear monotone operators, 
was first proved by  Brezis and Haraux in
 \cite[Proposition~2]{BrezisHaraux} in a Hilbert
space setting. A different proof was provided in \cite[Theorem~4.12]{BBW}. 
Let us now generalize to Banach spaces
(We mention that most of the proof of \ref{RecC:1}$\Leftrightarrow$\ref{RecC:2}
follows along the lines of
\cite[Theorem~4.12(i)$\Leftrightarrow$(ii)]{BBW}.)

\begin{proposition}[characterizations of rectangularity]\label{Para:L1}
Let  $A:X\rightarrow X^*$ be  continuous, linear, and monotone.
Then the following are equivalent: 
\begin{enumerate}
\item\label{RecC:1} $A$ is rectangular.
\item \label{RecC:2} 
$(\exi \beta>0)(\forall x\in X)$
$\langle x, Ax\rangle\geq\beta\|Ax\|^2$, i.e.,
$A$ is $\beta$-cocoercive. 
\item\label{RecC:3} $A^*|_X$ is rectangular.
\end{enumerate}
If $X$ is a real Hilbert space, then 
{\rm \ref{RecC:1}--\ref{RecC:3}} are also equivalent to either of the
following: 
\begin{enumerate}[resume]
\item \label{RecC:4}
$(\exi \gamma>0)$
$\|\gamma A-\Id\|\leq 1$.
\item \label{RecC:5}
$(\exi \beta>0)$
$A^{-1}-\beta\Id$ is monotone. 
\end{enumerate}
\end{proposition}
\begin{proof}
``\ref{RecC:1}$\Rightarrow$\ref{RecC:2}'': 
Set
\begin{equation}
f\colon X\to\RX\colon 
x\mapsto F_A(x,0).
\end{equation}
Note that $X\times \{0\}\subseteq\dom F_A$ by Lemma~\ref{lemR:1}. 
Thus, by Fact~\ref{pheps:1}, $f$ is continuous.
Since $(0,0)\in\gra A$,  Fact~\ref{f:Fitz} now yields 
$f(0)=F_A(0,0)=\langle 0,0\rangle=0$.
In view of the continuity of $f$ at $0$, there
exist $\alpha>0$ and $\delta>0$  such that
\begin{equation}
\big(\forall y\in\delta B_X\big)\quad
\sup_{a\in X}\big(\langle y, Aa\rangle-\langle a, Aa\rangle\big)
=F_A(y,0)=f(y)\leq \alpha.
\end{equation}
Hence $(\forall y\in\delta B_X)(\forall a\in X)$
$\langle y, Aa\rangle \leq \alpha+\langle a, Aa\rangle$.
Thus,
\begin{equation}
\label{BrHa:1}
(\forall a\in X)\quad 
\delta\| Aa\|=\sup\langle \delta B_X, Aa\rangle\leq
 \alpha+\langle a, Aa\rangle.
\end{equation}
Replacing $a$ by $\rho a$ and invoking the linearity of $A$, we obtain
\begin{subequations}
\label{BrHa:2}
\begin{align}
(\forall a\in X)(\forall \rho\in\RR)\quad
0&\leq-|\rho|\delta\| Aa\| +\alpha+\rho^2\langle a, Aa\rangle\\
&\leq-\rho\delta\| Aa\|
+ \alpha+\rho^2\langle a, Aa\rangle.
\end{align}
\end{subequations}
Then, by Fact~\ref{FQ:1} and \eqref{BrHa:2}, it follows that 
\begin{equation}
(\forall a\in X)\quad 
\delta^2\| Aa\|^2\leq 4\alpha\langle a, Aa\rangle. 
\end{equation}
The desired conclusion holds with 
$\beta:={\delta^2}/{(4\alpha)}$. 

``\ref{RecC:2}$\Rightarrow$\ref{RecC:1}'': 
Let $x\in X$. Then 
\begin{subequations}
\begin{align}
F_A(x,0)&=\sup_{a\in X}\big(
\langle x, Aa\rangle-\langle a, Aa\rangle\big)\\
&\leq\sup_{a\in X}\big(\| x\|\cdot \|Aa\|-\beta\|Aa\|^2\big)\\
&\leq\sup_{t\in\RR}\big(t\| x\|-\beta t^2\big)\\
&= \frac{\|x\|^2}{4\beta}.
\end{align}
\end{subequations}
Hence $X\times \{0\}\subseteq \dom F_A$ and the conclusion
follows from Lemma~\ref{lemR:1}.

``\ref{RecC:3}$\Leftrightarrow$\ref{RecC:1}'':  
Corollary~\ref{LemmaCP:3}. 
Now assume that $X$ is a real Hilbert space.
The proof concludes as follows. 
``\ref{RecC:1}$\Leftrightarrow$\ref{RecC:4}'': 
\cite[Theorem~4.12]{BBW}.
``\ref{RecC:2}$\Leftrightarrow$\ref{RecC:5}'': 
\cite[Example~22.6]{BC2011}.
\end{proof}

\begin{remark}
Some comments regarding Proposition~\ref{Para:L1} are in order. 
\begin{enumerate}
\item
When $X$ is a real Hilbert space and $A^*=A\neq 0$,
then \ref{RecC:2} holds with $\beta = 1/\|A\|$;
see \cite[Corollary~18.17]{BC2011} or
\cite[Corollary~3.4]{BC2010}. 
\item 
The condition in \ref{RecC:5} is also known as
\emph{strong monotonicity} of $A^{-1}$ with constant $\beta$.
The constant $\beta$ in item~\ref{RecC:2} is indeed the same as
the one in item~\ref{RecC:5}.
\end{enumerate}
\end{remark}

Let us now provide some ``bad'' operators
(see Section~\ref{s:id-t} for some ``good'' operators).

\begin{example}
[Volterra operator is neither rectangular nor paramonotone]
\label{ex:Volterra}\ \\
{\rm (See also \cite[Example~3.3]{BBW} and \cite[Problem~148]{Halmos}.)}
Suppose that $X=L^2[0,1]$.
Recall that 
the \emph{Volterra integration
operator} is defined by
\begin{equation} \label{e:aug17:a}
V \colon X \to X \colon x \mapsto Vx, \quad\text{where}\quad
Vx\colon[0,1]\to\RR\colon t\mapsto \int_{0}^{t}x,
\end{equation}
and that its adjoint is given by
\begin{equation}
V^* \colon X \to X \colon x \mapsto V^*x, \quad\text{where}\quad
V^*x\colon[0,1]\to\RR\colon t\mapsto \int_{t}^{1}x.
\end{equation}
Then
$V$ is neither paramonotone nor rectangular; 
consequently, $V^*$ is
neither paramonotone nor rectangular.
\end{example}
\begin{proof} 
Set $e\equiv 1\in X$.
Then 
by \cite[Example~3.3]{BBW}, 
$V$ is a continuous, injective, linear, and maximally monotone,
with symmetric part $V_+\colon x\mapsto \thalb\scal{x}{e}e$. 
Now pick 
$x\in \{e\}^\perp\smallsetminus\{0\}$ 
(for instance, consider $t\mapsto t-\thalb$). 
Then 
 \begin{equation}
 \langle x,Vx\rangle=\langle x, V_+x\rangle=\big\langle x,\tfrac{1}{2}\langle e,x\rangle e\big\rangle
 =\tfrac{1}{2}\langle e,x\rangle^2=0.
 \end{equation}
However, since $x\neq 0$ and $V$ is injective, we have
$Vx\neq0=V0$. Thus $V$ cannot be paramonotone.  
By Proposition~\ref{PropRDP:1}, $V$ is not rectangular.  
It follows from Proposition~\ref{Para:L1} and Proposition~\ref{LemmaCP:2}
that $V^*$ is neither paramonotone nor rectangular.
\end{proof}

The next result provides a mechanism for constructing
paramonotone linear operators that are not rectangular.

\begin{proposition}\label{PExNR:1}
Suppose that $X$ is reflexive.
Let $A$ and $B$ be continuous, linear, and monotone on $X$.
Suppose that 
$A$ is paramonotone and $\ran A$ is a proper dense subspace of $X$, and
that $B$ is bijective and  skew. 
Then $A+B$ is maximally monotone, strictly monotone and paramonotone,
but not rectangular.
\end{proposition}
\begin{proof}
Set $C := A+B$. 
By Fact~\ref{F:1}, $A$, $B$, and $C$  are maximally monotone, linear, and
continuous. 
Since $\ran A$ is proper dense, 
Fact~\ref{Rea:1}\ref{Th:32s} and Fact~\ref{KerRa:1} imply that
\begin{equation}
\ker A=(\ran A^*)^{\bot}=(\ran A)^{\bot}=\{0\}.\label{PExNR:Ea1}
\end{equation}
Since $B$ is skew, for every $x\in X$,
\begin{equation}
\langle Cx, x\rangle=\langle (A+B)x, x\rangle=\langle Ax, x\rangle.
\end{equation}
Using \eqref{PExNR:Ea1} and the paramonotonicity of $A$, we obtain
\begin{equation}
(\forall x\in X)\quad 
\langle Cx, x\rangle=0\Leftrightarrow Ax=0\Leftrightarrow x=0.
\end{equation}
Therefore, $C$ is strictly and thus paramonotone. 

By the Bounded Inverse Theorem, 
$B^{-1}$ is a bounded linear operator; thus,
there exists $\gamma>0$ such that
\begin{equation}
\label{PExNR:E2}
(\forall x\in X)\quad 
\|Bx\|\geq \gamma \|x\|.
\end{equation}

We now claim that  there exists a sequence $(x_n)_{n\in\NN}$
such that
\begin{align}
(\forall\nnn)\quad \|x_n\|=1,
\quad\text{while}\quad Ax_n\to 0.\label{PExNR:E1}
\end{align}
Since $\ran A$ is not closed, we take 
$z^*\in\overline{\ran A} \smallsetminus \ran A$.  
Thus there exists a sequence $(y_n)_{n\in\NN}$ in $X$ such that
\begin{equation}
Ay_n\to z^*.\label{PNR:s1}
\end{equation}
Let us show that $\|y_n\|\to\pinf$. 
Otherwise, $(y_n)_\nnn$ possesses a 
a weakly convergent subsequence---which for convenience we still 
denote by $(y_n)_{n\in\NN}$--- say, $y_n\weakly y\in X$.  
By \cite[Theorem~3.1.11]{Megg}, $Ay_n\weakly Ay$. 
Combining with \eqref{PNR:s1}, we deduce that 
$z^*=Ay$, which contradicts our assumption that $z^*\notin\ran A$.
Hence $\|y_n\|\to\pinf$. We assume that
$(\forall\nnn)$ $y_n\neq 0$ and we set $x_n := y_n/\|y_n\|$.
By \eqref{PNR:s1}, $Ax_n\to 0$. This completes the proof of 
\eqref{PExNR:E1}. 

Now suppose that $\beta\geq 0$ satisfies
\begin{equation}
(\forall x\in X)\quad 
\langle Cx,x\rangle\geq\beta\|Cx\|^2.
\label{Expa:1}
\end{equation}
Using \eqref{Expa:1}, we estimate that for every $\nnn$ 
\begin{subequations}
\begin{align}
\|Ax_n\|&=\|Ax_n\|\cdot\|x_n\|\geq\langle Ax_n,x_n\rangle=\langle
Cx_n,x_n\rangle\\
&\geq\beta\|Cx_n\|^2=\beta\|(A+B)x_n\|^2\\
&\geq \beta\big(\|Bx_n\|-\|Ax_n\|\big)^2. 
\end{align}
\end{subequations}
Taking $\varliminf$ and recalling 
\eqref{PExNR:E1}\&\eqref{PExNR:E2},
we deduce that
$0\geq\beta\gamma^2$ and hence $\beta = 0$.
Therefore, by Proposition~\ref{Para:L1}, 
$C$ cannot be rectangular. 
\end{proof}

The following is an incarnation of Proposition~\ref{PExNR:1} in
the Hilbert space of real square-summable sequences. It provides
a counterexample complementary to Example~\ref{ExaRniP}. 

\begin{example}[paramonotone $\not\Rightarrow$ rectangular]
\label{ExNR:1}
Suppose that $X=\ell^2$, and define two 
continuous linear maximally monotone operators on $X$ via 
\begin{equation}
A\colon X\to X\colon
(x_n)_\nnn\mapsto \big(\tfrac{1}{n}x_n\big)_\nnn
\end{equation}
and 
\begin{equation}
B\colon X\to X\colon
(x_n)_\nnn\mapsto (-x_2,x_1,-x_4,x_3,\ldots).
\end{equation}
Then $A+B$ is maximally monotone, strictly monotone and paramonotone, but
not rectangular. 
\end{example}

\section{Displacement mappings}

\label{s:id-t}

For certain maximally monotone operators in Hilbert space, 
we are able to obtain
sharper conclusions as the following result illustrates.
Let us assume that $X$ is a Hilbert space, and let 
$T\colon X\to X$. Recall that $T$ is \emph{nonexpansive}
if 
\begin{equation}
(\forall x\in X)(\forall y\in X)\quad
\|Tx-Ty\|\leq \|x-y\|;
\end{equation}
$T$ is \emph{firmly nonexpansive}
if 
\begin{equation}
(\forall x\in X)(\forall y\in X)\quad
\|Tx-Ty\|^2\leq\scal{x-y}{Tx-Ty};
\end{equation}
$T$ is \emph{$\thalb$-cocoercive} if
$\thalb T$ is firmly nonexpansive. 
It is straightforward to show that $T$ is firmly nonexpansive
if and only if $2T-\Id$ is nonexpansive. 
A deeper result due to Minty \cite{Minty} 
(see, e.g., \cite{Minty}, \cite{EckBer}, and
\cite[Chapter~23]{BC2011}) 
states that the class of  firmly nonexpansive operators on $X$ 
coincides with 
the class of resolvents $J_A = (\Id+A)^{-1}$,
where $A$ is an arbitrary maximally monotone
operator on $X$.

\begin{theorem}[displacement mapping]
\label{t:frozen}
Suppose that $X$ is a Hilbert space,
let $T\colon X\to X$ be 
nonexpansive,  
and define the corresponding \emph{displacement mapping} by 
\begin{equation}
A = \Id-T.
\end{equation}
Then the following hold:
\begin{enumerate}
\item 
\label{t:frozen1}
$A$ is maximally monotone.
\item 
\label{t:frozen2}
$A$ is rectangular. 
\item 
\label{t:frozen3}
$A$ is $\thalb$-cocoercive.
\item 
\label{t:frozen4}
$A^{-1}$ is strongly monotone with constant $\thalb$. 
\item
\label{t:frozen5}
$A^{-1}$ is strictly monotone.
\item
\label{t:frozen6}
$A$ is paramonotone. 
\end{enumerate}
\end{theorem}
\begin{proof}
\ref{t:frozen1}:
The maximal monotonicity of $A$ is well known;
see, e.g., \cite[Example~20.26]{BC2011}.
\ref{t:frozen2}\&\ref{t:frozen3}:
Since $T$ is nonexpansive,
\cite[Proposition~4.2 and Proposition~23.7]{BC2011} yield
the existence of a maximally monotone operator $B$ on $X$ such that
$T = 2J_B-\Id$, where $J_B = (\Id+B)^{-1}$ denotes the resolvent of $B$.
Thus $A = \Id - T = 2(\Id-J_B)=2J_{B^{-1}}$ by 
\cite[Proposition~23.18]{BC2011}. 
Now \cite[Example~24.16]{BC2011} implies the
rectangularity of $J_{B^{-1}}$.
Thus, by \cite[Proposition~24.15(ii)]{BC2011}, 
$A= 2J_{B^{-1}}$ is rectangular as well. 
Hence $\thalb A$ is firmly nonexpansive, i.e.,
$A$ is $\thalb$-cocoercive. 
\ref{t:frozen4}: 
It follows from \ref{t:frozen3} and 
\cite[Example~22.6]{BC2011} that $A^{-1}$ is
strongly monotone with constant $\thalb$. 
\ref{t:frozen5}: Immediate from \ref{t:frozen4}. 
\ref{t:frozen6}: 
Let $x$ and $y$ be in $X$ such that
$0=\scal{x-y}{Ax-Ay}=\|x-y\|^2-\scal{x-y}{Tx-Ty}$.
Then $\|x-y\|^2 =\scal{x-y}{Tx-Ty}\leq\|x-y\|\cdot\|Tx-Ty\|
\leq \|x-y\|^2$.
By Cauchy-Schwarz, $x-y=Tx-Ty$, i.e,. $Ax=Ay$. 
Thus, $A$ is paramonotone.
(Alternatively, one may argue as follows:
$A^{-1}$ is strictly monotone by \ref{t:frozen6}
and hence paramonotone. Thus $(A^{-1})^{-1} = A$ is paramonotone 
as well, as the inverse of a paramonotone operator.)
\end{proof}

The displacement mapping in the following example is 
generally not a subdifferential
operator (unless $m=2$); however, it is rectangular and paramonotone.
Moreover, the 
rectangularity of the displacment mapping 
played a crucial role in \cite{Bau:1} and \cite{BMMW}.

\begin{example}[cyclic right-shift operator]
Let $Y$ be a real Hilbert space, let $m\in\NN$, and suppose
that $X$ is the Hilbert product space $Y^m$. 
Define the \emph{cyclic right-shift operator}
$R$ by
\begin{equation}
R\colon X\to X \colon (x_1,x_2,\dots,x_m)
\mapsto (x_m, x_1,\ldots,x_{m-1}).
\end{equation} 
Then $\Id-R$ is rectangular and paramonotone.
\end{example}
\begin{proof}
Since $(\forall x\in X)$ $\|Rx\|=\|x\|$,
we deduce that $R$ is nonexpansive.
Now apply Theorem~\ref{t:frozen}. 
\end{proof}

\section*{Acknowledgments}
Heinz Bauschke was partially supported by the Natural Sciences and
Engineering Research Council of Canada and
by the Canada Research Chair Program.
Xianfu Wang was partially supported by the Natural
Sciences and Engineering Research Council of Canada.


\end{document}